\documentclass[12pt]{amsart}
\usepackage{amsmath,amsthm,amsfonts,amssymb,mathrsfs}
\date{\today}

\usepackage{color}

\usepackage{hyperref}

\newtheorem{theorem}{Theorem}

\newtheorem{proposition}[theorem]{Proposition}
\newtheorem{corollary}[theorem]{Corollary}
\newtheorem{lemma}[theorem]{Lemma}
\theoremstyle{definition}
\newtheorem{example}[theorem]{Example}
\newtheorem{remark}[theorem]{Remark}

\begin{document}

\title[On semitopological interassociates of the bicyclic monoid]{On semitopological interassociates of the bicyclic monoid}

\author[Oleg~Gutik and Kateryna Maksymyk]{Oleg~Gutik and Kateryna Maksymyk}
\address{Faculty of Mathematics, National University of Lviv,
Universytetska 1, Lviv, 79000, Ukraine}
\email{o\underline{\hskip5pt}\,gutik@franko.lviv.ua, ovgutik@yahoo.com, kate.maksymyk15@gmail.com}



\begin{abstract}
Semitopological interassociates $\mathscr{C}_{m,n}$ of the bicyclic semigroup $\mathscr{C}(p,q)$ are studied. In particular, we show that for arbitrary non-negative integers $m$, $n$ and every Hausdorff topology $\tau$ on $\mathscr{C}_{m,n}$ such that $\left(\mathscr{C}_{m,n},\tau\right)$ is a semitopological semigroup, is discrete. Also, we prove that if an interassociate of the bicyclic monoid $\mathscr{C}_{m,n}$ is a dense subsemigroup of a Hausdorff semitopological semigroup $(S,\cdot)$ and $I=S\setminus\mathscr{C}_{m,n}\neq\varnothing$ then $I$ is a two-sided ideal of the semigroup $S$ and show that for arbitrary non-negative integers $m$, $n$, any Hausdorff locally compact semitopological semigroup $\mathscr{C}_{m,n}^0=\mathscr{C}_{m,n}\sqcup\{0\}$ is either discrete or compact.
\end{abstract}

\maketitle

We shall follow the terminology of \cite{Carruth-Hildebrant-Koch-1983-1986, Clifford-Preston-1961-1967, Engelking-1989, Ruppert-1984}. In this paper all spaces will be assumed to be Hausdorff. By $\mathbb{N}_0$ and $\mathbb{N}$ we denote the sets of non-negative integers and positive integers, respectively. If $A$ is a subset of a topological space $X$  then by $\operatorname{cl}_X(A)$ and $\operatorname{int}_X(A)$ we denote the closure and interior of $A$ in $X$, respectively.

A \emph{semigroup} is a non-empty set with a binary associative
operation.

The \emph{bicyclic semigroup} (or the \emph{bicyclic monoid}) ${\mathscr{C}}(p,q)$ is the semigroup with
the identity $1$ generated by two elements $p$ and $q$ subject
only to the condition $pq=1$. The bicyclic monoid ${\mathscr{C}}(p,q)$ is a combinatorial bisimple $F$-inverse semigroup (see \cite{Lawson-1998}) and it plays an important role in the algebraic theory of semigroups and in the theory of topological semigroups.
For example the well-known O.~Andersen's result~\cite{Andersen-1952}
states that a ($0$--)simple semigroup is completely ($0$--)simple
if and only if it does not contain the bicyclic semigroup. The
bicyclic semigroup cannot be embedded into the stable
semigroups~\cite{Koch-Wallace-1957}.

An interassociate of a semigroup $(S,\cdot)$ is a semigroup $(S,\ast)$ such that for all $a,b,c\in S$, $a\cdot(b\ast c)=(a\cdot b)\ast c$ and $a\ast(b\cdot c)=(a\ast b)\cdot c$. This definition of interassociativity was studied extensively in 1996 by Boyd et al \cite{Boyd-Gould-Nelson-1998}. Certain classes of semigroups are known to give rise to interassociates with various properties. For example, it is very easy to show that if $S$ is a monoid, every interassociate must satisfy the condition $a\ast b = acb$ for some fixed element $c\in S$ (see \cite{Boyd-Gould-Nelson-1998}). This type of interassociate
was called a variant by Hickey \cite{Hickey-1983}. In addition, every interassociate of a completely simple semigroup is completely simple \cite{Boyd-Gould-Nelson-1998}. Finally, it is relatively easy to show that every interassociate of a group is isomorphic to the group itself.

In the paper \cite{Givens-Rosin-Linton-2016??} the bicyclic semigroup ${\mathscr{C}}(p,q)$ and its interassociates are investigated. In particular, if $p$ and $q$ are generators of the bicyclic semigroup ${\mathscr{C}}(p,q)$ and $m$ and $n$ are fixed nonnegative integers, the operation $a\ast_{m,n}b=aq^mp^nb$ is known to be an interassociate. It was shown that for distinct pairs $(m, n)$ and $(s, t)$, the interassociates $({\mathscr{C}}(p,q),\ast_{m,n})$ and $(\mathscr{C}(p,q),\ast_{s,t})$ are not isomorphic. Also in \cite{Givens-Rosin-Linton-2016??} the authors generalized a result regarding homomorphisms on ${\mathscr{C}}(p,q)$ to homomorphisms on its interassociates.

Later for fixed non-negative integers $m$ and $n$ the interassociate $({\mathscr{C}}(p,q),\ast_{m,n})$ of the bicyclic monoid $\mathscr{C}(p,q)$ will be denoted by $\mathscr{C}_{m,n}$.

A ({\it semi})\emph{topological semigroup} is a topological space with a (separately) continuous semigroup operation.

The bicyclic semigroup admits only the discrete semigroup topology and if a topological semigroup $S$ contains it as a dense subsemigroup then ${\mathscr{C}}(p,q)$ is an open subset of $S$~\cite{Eberhart-Selden-1969}. Bertman and  West in \cite{Bertman-West-1976} extend this result for the case of Hausdorff semitopological semigroups. Stable and $\Gamma$-compact topological semigroups do not contain the bicyclic semigroup~\cite{Anderson-Hunter-Koch-1965, Hildebrant-Koch-1986}. The problem of an embedding of the bicyclic monoid into compact-like topological semigroups studied in \cite{Banakh-Dimitrova-Gutik-2009, Banakh-Dimitrova-Gutik-2010, Gutik-Repovs-2007}. Also in the paper \cite{Fihel-Gutik-2011} it was proved that the discrete topology is the unique topology on the extended bicyclic semigroup $\mathscr{C}_{\mathbb{Z}}$ such that the semigroup operation on $\mathscr{C}_{\mathbb{Z}}$ is separately continuous. Amazing dichotomy for the bicyclic monoid with adjoined zero $\mathscr{C}^0={\mathscr{C}}(p,q)\sqcup\{0\}$ was proved in \cite{Gutik-2015}: every Hausdorff locally compact semitopological bicyclic semigroup with adjoined zero $\mathscr{C}^0$ is either compact or discrete.

In this paper we study semitopological interassociates $({\mathscr{C}}(p,q),\ast_{m,n})$ of the bicyclic monoid $\mathscr{C}(p,q)$ for arbitrary non-negative integers $m$ and $n$. Some results from \cite{Bertman-West-1976, Eberhart-Selden-1969, Gutik-2015} obtained for the bicyclic semigroup are extended to its interassociate $({\mathscr{C}}(p,q),\ast_{m,n})$. In particular, we show that for arbitrary non-negative integers $m$, $n$ and every Hausdorff topology $\tau$ on $\mathscr{C}_{m,n}$ such that $\left(\mathscr{C}_{m,n},\tau\right)$ is a semitopological semigroup, is discrete. Also, we prove that if an  interassociate of the bicyclic monoid $\mathscr{C}_{m,n}$ is a dense subsemigroup of a Hausdorff semitopological semigroup $(S,\cdot)$ and $I=S\setminus\mathscr{C}_{m,n}\neq\varnothing$ then $I$ is a two-sided ideal of the semigroup $S$ and show that for arbitrary non-negative integers $m$, $n$, any Hausdorff locally compact semitopological semigroup $\mathscr{C}_{m,n}^0$ ($\mathscr{C}_{m,n}^0=\mathscr{C}_{m,n}\sqcup\{0\}$) is either discrete or compact.


\bigskip

For arbitrary $m,n\in N$ we denote
\begin{equation*}
\mathscr{C}_{m,n}^*=\left\{q^{n+k}p^{m+l}\in\mathscr{C}_{m,n}\colon k.l\in \mathbb{N}_0\right\}.
\end{equation*}
The semigroup operation $\ast_{m,n}$ of $\mathscr{C}_{m,n}$ implies that $\mathscr{C}_{m,n}^*$ is a subsemigroup of $\mathscr{C}_{m,n}$.

We need the following trivial lemma.

\begin{lemma}\label{lemma-1}
For arbitrary non-negative integers $m$ and $n$ the subsemigroup $\mathscr{C}_{m,n}^*$ of $\mathscr{C}_{m,n}$ is isomorphic to the bicyclic semigroup $\mathscr{C}(p,q)$ under the map $\iota\colon\mathscr{C}(p,q)\to \mathscr{C}_{m,n}^*\colon q^ip^j\mapsto q^{n+i}p^{m+j}$, $i,j\in \mathbb{N}_0$.
\end{lemma}

\begin{proof}
It is sufficient to show that the map $\iota\colon\mathscr{C}(p,q)\to \mathscr{C}_{m,n}^*$ is a homomorphism, because $\iota$ is bijective. Then for arbitrary $i,j,k,l\in \mathbb{N}_0$ we have that
\begin{equation*}
  \iota(q^i p^j\cdot q^k p^l)=
  \left\{
    \begin{array}{ll}
      \iota(q^{i-j+k}p^l), & \hbox{if~} j<k;\\
      \iota(q^ip^{j-k+l}), & \hbox{if~} j\geqslant k
    \end{array}
  \right.=
  \left\{
    \begin{array}{ll}
      q^{n+i-j+k}p^{m+l}, & \hbox{if~} j<k;\\
      q^{n+i}p^{m+j-k+l}, & \hbox{if~} j\geqslant k
    \end{array}
  \right.
\end{equation*}
and
\begin{equation*}
\begin{split}
  \iota(q^ip^j)\ast_{m,n} \iota(q^kp^l) & =q^{n+i}p^{m+j}\ast_{m,n} q^{n+k}p^{m+l}=\\
    & = q^{n+i}p^{m+j}\cdot q^mp^n\cdot q^{n+k}p^{m+l}=\\
    & =
  q^{n+i}p^j\cdot q^kp^{m+l}= \\
    & =
      \left\{
    \begin{array}{ll}
      q^{n+i-j+k}p^{m+l}, & \hbox{if~} j<k;\\
      q^{n+i}p^{m+j-k+l}, & \hbox{if~} j\geqslant k,
    \end{array}
  \right.
\end{split}
\end{equation*}
which completes the proof of the lemma.
\end{proof}

Lemma~I.1 from \cite{Eberhart-Selden-1969} and the definition of the semigroup operation in $\mathscr{C}_{m,n}$ imply the following:

\begin{lemma}\label{lemma-2}
For arbitrary non-negative integers $m$ and $n$  and for each elements $a,b\in\mathscr{C}_{m,n}$ both sets
\begin{equation*}
  \left\{x\in\mathscr{C}_{m,n}\colon a\ast_{m,n}x=b\right\} \qquad \hbox{and} \qquad \left\{x\in\mathscr{C}_{m,n}\colon x\ast_{m,n}a=b\right\}
\end{equation*}
are finite; that is, both left  and right translation by $a$ are finite-to-one maps.
\end{lemma}

The following theorem generalizes  the Eberhart--Selden result on semigroup topologization of the bicyclic semigroup (see \cite[Corollary~I.1]{Eberhart-Selden-1969}) and the corresponding statement for the case semitopological semigroups in \cite{Bertman-West-1976}.

\begin{theorem}\label{theorem-3}
For arbitrary non-negative integers $m$, $n$, every Hausdorff semitopological semigroup $\left(\mathscr{C}_{m,n},\tau\right)$ is discrete.
\end{theorem}

\begin{proof}
By Proposition~1 of \cite{Bertman-West-1976} every Hausdorff semitopological semigroup $\mathscr{C}(p,q)$ is discrete. Hence Lemma~\ref{lemma-1} implies that for any element $x\in\mathscr{C}_{m,n}^*$ there exists an open neighbourhood $U(x)$ of the point $x$ in $\left(\mathscr{C}_{m,n},\tau\right)$ such that $U(x)\cap \mathscr{C}_{m,n}^*=\left\{x\right\}$. Fix an arbitrary open neighbourhood $U(q^np^m)$ of the point $q^np^m$ in $\left(\mathscr{C}_{m,n},\tau\right)$ such that $U(q^np^m)\cap \mathscr{C}_{m,n}^*=\left\{q^np^m\right\}$. Then the separate continuity of the semigroup operation in $\left(\mathscr{C}_{m,n},\tau\right)$ implies that there exists an open neighbourhood $V(q^np^m)\subseteq U(q^np^m)$ of the point $q^np^m$ in the space $\left(\mathscr{C}_{m,n},\tau\right)$ such that
\begin{equation*}
  V(q^np^m)\ast_{m,n} q^np^m\subseteq U(q^np^m) \qquad \hbox{and} \qquad q^np^m\ast_{m,n} V(q^np^m)\subseteq U(q^np^m).
\end{equation*}
Suppose to the contrary that the neighbourhood $V(q^np^m)$ is an infinite set. Then at least one of the following conditions holds:
\begin{itemize}
  \item[$(i)$] there exists a non-negative integer $i_0<n$ such that the set $A=\left\{q^{i_0}p^l\colon l\in N\right\}\cap V(q^np^m)$ is infinite;
  \item[$(ii)$] there exists a non-negative integer $j_0<m$ such that the set $B=\left\{q^{l}p^{j_0}\colon l\in N\right\}\cap V(q^np^m)$ is infinite.
\end{itemize}
In case $(i)$ for arbitrary $q^{i_0}p^l\in A$ we have that
\begin{equation*}
\begin{split}
  q^np^m\ast_{m,n}q^{i_0}p^l & =q^np^mq^mp^nq^{i_0}p^l=q^np^nq^{i_0}p^l= \\
    & = q^np^{n-i_0+l}\notin U(q^np^m) \quad \hbox{~for sufficiently large~} l;
\end{split}
\end{equation*}
and similarly in case $(ii)$ we obtain that
\begin{equation*}
\begin{split}
  q^{l}p^{j_0}\ast_{m,n} q^np^m & =q^{l}p^{j_0} q^mp^nq^np^m=q^{l}p^{j_0} q^mp^m= \\
    & =  q^{m-j_0+l}p^m\notin U(q^np^m) \quad \hbox{~for sufficiently large~} l;
\end{split}
\end{equation*}
for each $q^{l}p^{j_0}\in B$, which contradicts the separate continuity of the semigroup operation in $\left(\mathscr{C}_{m,n},\tau\right)$. The obtained contradiction implies that $q^np^m$ is an isolated point in the space $\left(\mathscr{C}_{m,n},\tau\right)$.

Now, since the semigroup $\mathscr{C}_{m,n}$ is simple (see \cite[Section~2]{Givens-Rosin-Linton-2016??}) for arbitrary $a,b\in \mathscr{C}_{m,n}$ there exist $x,y\in \mathscr{C}_{m,n}$ such that $xay=b$. The above argument implies that for arbitrary element $u\in \mathscr{C}_{m,n}$ there exist $x_u,y_u\in \mathscr{C}_{m,n}$ such that $x_u uy_u=q^np^m$. Now, by Lemma~\ref{lemma-2} we get that the equation $x_uxy_u=q^np^m$ has finitely many solutions. This and the separate continuity of the semigroup operation in $\left(\mathscr{C}_{m,n},\tau\right)$ imply that the point $u$ has an open finite neighbourhood in $\left(\mathscr{C}_{m,n},\tau\right)$, and hence, by the Hausdorffness of $\left(\mathscr{C}_{m,n},\tau\right)$,  $u$ is an isolated point in $\left(\mathscr{C}_{m,n},\tau\right)$. Then the choice of $u$ implies that all elements of the semigroup $\mathscr{C}_{m,n}$ are isolated points in $\left(\mathscr{C}_{m,n},\tau\right)$.
\end{proof}

The following theorem generalizes Theorem~I.3 from \cite{Eberhart-Selden-1969}.

\begin{theorem}\label{theorem-4}
If $m$ and $n$ are arbitrary non-negative integers, the interassociate $\mathscr{C}_{m,n}$ of the bicyclic monoid ${\mathscr{C}}(p,q)$ is a dense subsemigroup of a Hausdorff semitopological semigroup $(S,\cdot)$, and $I=S\setminus\mathscr{C}_{m,n}\neq\varnothing$ then $I$ is a two-sided ideal of the semigroup $S$.
\end{theorem}

\begin{proof}
Fix an arbitrary element $y\in I$. If $x\cdot y=z\notin I$ for some $x\in\mathscr{C}_{m,n}$ then there exists an open neighbourhood $U(y)$ of the point $y$ in the space $S$ such that $\{x\}\cdot U(y)=\{z\}\subset\mathscr{C}_{m,n}$. The neighbourhood $U(y)$ contains infinitely many elements of the semigroup $\mathscr{C}_{m,n}$ which contradicts Lemma~\ref{lemma-2}. The obtained contradiction implies that $x\cdot y\in I$ for all $x\in \mathscr{C}_{m,n}$ and $y\in I$. The proof of the statement that $y\cdot x\in I$ for all $x\in\mathscr{C}_{m,n}$ and $y\in I$ is similar.

Suppose to the contrary that $x\cdot y=w\notin I$ for some $x,y\in I$. Then $w\in \mathscr{C}_{m,n}$ and the separate continuity of the semigroup operation in $S$ implies that there exist open neighbourhoods $U(x)$ and $U(y)$ of the points $x$ and $y$ in $S$, respectively, such that $\{x\}\cdot U(y)=\{w\}$ and $U(x)\cdot \{y\}=\{w\}$. Since both neighbourhoods $U(x)$ and $U(y)$ contain infinitely many elements of the semigroup $\mathscr{C}_{m,n}$, both equalities $\{x\}\cdot U(y)=\{w\}$ and $U(x)\cdot \{y\}=\{w\}$ contradict the mentioned above part of the proof, because $\{x\}\cdot \left(U(y)\cap\mathscr{C}_{m,n}\right)\subseteq I$. The obtained contradiction implies that $x\cdot y\in I$.
\end{proof}

We recall that a topological space X is said to be:
\begin{itemize}
    \item \emph{compact} if every open cover of $X$ contains a finite subcover;
    \item \emph{countably compact} if each closed discrete subspace of $X$ is finite;
    \item \emph{feebly compact} if each locally finite open cover of $X$ is finite;
    \item \emph{pseudocompact} if $X$ is Tychonoff and each continuous real-valued function on $X$ is bounded;
    \item \emph{locally compact} if every point $x$ of $X$ has an open neighbourhood $U(x)$ with the compact closure $\operatorname{cl}_X(U(x))$;
    \item \emph{\v{C}ech-complete} if $X$ is Tychonoff and there exists a compactification $cX$ of $X$ such that the remainder of $X$ is an $F_\sigma$-set in $cX$.
\end{itemize}
According to Theorem~3.10.22 of \cite{Engelking-1989}, a Tychonoff topological space $X$ is feebly compact if and only if $X$ is pseudocompact. Also, a Hausdorff topological space $X$ is feebly compact if and only if every locally finite family of non-empty open subsets of $X$ is finite. Every compact space and every sequentially compact space are countably compact, every countably compact space is feebly compact (see \cite{Arkhangelskii-1992}).

A topological semigroup $S$ is called
\emph{$\Gamma$-compact} if for every $x\in S$ the closure of the set
$\{x,x^2,x^3,\ldots\}$ is compact in $S$ (see
\cite{Hildebrant-Koch-1986}). Since by Lemma~\ref{lemma-1} the semigroup
$\mathscr{C}_{m,n}$ contains the bicyclic semigroup as a subsemigroup the results obtained in
\cite{Anderson-Hunter-Koch-1965}, \cite{Banakh-Dimitrova-Gutik-2009},
\cite{Banakh-Dimitrova-Gutik-2010}, \cite{Gutik-Repovs-2007},
\cite{Hildebrant-Koch-1986} imply the following corollary

\begin{corollary}\label{corollary-5}
Let $m$ and $n$ be arbitrary non-negative integers. If a Hausdorff topological semigroup $S$ satisfies one of the following conditions:
\begin{itemize}
  \item[$(i)$] $S$ is compact;
  \item[$(ii)$] $S$ is $\Gamma$-compact;
  \item[$(iii)$] the square $S\times S$ is countably compact; or
  \item[$(iv)$] the square $S\times S$ is a Tychonoff pseudocompact space,
\end{itemize}
then $S$ does not contain the semigroup $\mathscr{C}_{m,n}$.
\end{corollary}

\begin{proposition}\label{proposition-6}
Let $m$ and $n$ be arbitrary non-negative integers. Let $S$ be a Hausdorff topological semigroup which contains a dense subsemigroup $\mathscr{C}_{m,n}$. Then for every $c\in\mathscr{C}_{m,n}$ the set
\begin{equation*}
    D_c=\left\{(x,y)\in\mathscr{C}_{m,n}\times\mathscr{C}_{m,n}\colon x\ast_{m,n} y=c\right\}
\end{equation*}
is an open-and-closed subset of $S\times S$.
\end{proposition}

\begin{proof}
By Theorem~\ref{theorem-3}, $\mathscr{C}_{m,n}$ is a discrete subspace of $S$ and hence Theorem~3.3.9 of \cite{Engelking-1989} implies that $\mathscr{C}_{m,n}$ is an open subspace of $S$. Then the continuity of the semigroup operation of $S$ implies that $D_c$ is an open subset of $S\times S$ for every $c\in\mathscr{C}_{m,n}$.

Suppose that there exists $c\in\mathscr{C}_{m,n}$ such that $D_c$ is a non-closed subset of $S\times S$. Then there exists an accumulation point $(a,b)\in S\times S$ of the set $D_c$. The continuity of the semigroup operation in $S$ implies that $a\cdot b=c$. But $\mathscr{C}_{m,n}\times \mathscr{C}_{m,n}$ is a
discrete subspace of $S\times S$ and hence by Theorem~\ref{theorem-4} the points $a$ and $b$ belong to the two-sided
ideal $I=S\setminus \mathscr{C}_{m,n}$ and hence the product $a\cdot b\in S\setminus \mathscr{C}_{m,n}$ cannot be equal to the element $c$.
\end{proof}

\begin{theorem}\label{theorem-7}
Let $m$ and $n$ be arbitrary non-negative integers. If a Hausdorff topological semigroup $S$ contains $\mathscr{C}_{m,n}$ as a dense subsemigroup then the square $S\times S$ is not feebly compact.
\end{theorem}

\begin{proof}
By Proposition~\ref{proposition-6} for every $c\in\mathscr{C}_{m,n}$ the square $S\times S$ contains an open-and-closed discrete subspace $D_c$.
In the case when $c=q^np^m$, the subspace $D_c$ contains an infinite subset $\left\{\left(q^np^{m+i},q^{n+i}p^m\right)\colon\right.$ $\left.i\in\mathbb{N}_0\right\}$ and hence $D_c$ is infinite. This implies that the square $S\times S$ is not feebly compact.
\end{proof}

For arbitrary non-positive integers $m$ and $n$ by $\mathscr{C}_{m,n}^0$ we denote the interassociate $\mathscr{C}_{m,n}$ with an adjoined zero $0$ of the bicyclic monoid ${\mathscr{C}}(p,q)$, i.e., $\mathscr{C}_{m,n}^0=\mathscr{C}_{m,n}\sqcup\left\{0\right\}$.

\begin{example}\label{example-7}
On the semigroup $\mathscr{C}_{m,n}^0$ we define a topology $\tau_{\operatorname{\textsf{Ac}}}$ in the following way:
\begin{itemize}
  \item[$(i)$] every element of the semigroup $\mathscr{C}_{m,n}$ is an isolated point in the space $\left(\mathscr{C}_{m,n}^0,\tau_{\operatorname{\textsf{Ac}}}\right)$;
  \item[$(ii)$] the family $\mathscr{B}(0)=\left\{U\subseteq \mathscr{C}_{m,n}^0\colon U\ni 0 \hbox{~and~} \mathscr{C}_{m,n}\setminus U \hbox{~is finite}\right\}$ determines a base of the topology $\tau_{\operatorname{\textsf{Ac}}}$ at zero $0\in\mathscr{C}_{m,n}^0$,
\end{itemize}
i.e., $\tau_{\operatorname{\textsf{Ac}}}$ is the topology of the Alexandroff one-point compactification of the discrete space $\mathscr{C}_{m,n}$ with the remainder $\{0\}$. The semigroup operation in $\left(\mathscr{C}_{m,n}^0,\tau_{\operatorname{\textsf{Ac}}}\right)$ is separately continuous, because all elements of the interassociate $\mathscr{C}_{m,n}$ of the bicyclic semigroup ${\mathscr{C}}(p,q)$ are isolated points in the space $(\mathscr{C}_{m,n}^0,\tau_{\operatorname{\textsf{Ac}}})$ and the left and right translations in the semigroup $\mathscr{C}_{m,n}$ are finite-to-one maps (see Lemma~\ref{lemma-2}).
\end{example}

\begin{remark}\label{remark-8}
By Theorem~\ref{theorem-3} the discrete topology $\tau_{\textsf{d}}$ is a unique Hausdorff topology on the interassociate $\mathscr{C}_{m,n}$ of the bicyclic monoid ${\mathscr{C}}(p,q)$, $m,n\in \mathbb{N}_0$, such that $\mathscr{C}_{m,n}$ is a semitopological semigroup. So $\tau_{\operatorname{\textsf{Ac}}}$ is the unique compact topology on $\mathscr{C}_{m,n}^0$ such that $(\mathscr{C}_{m,n}^0,\tau_{\operatorname{\textsf{Ac}}})$ is a Hausdorff compact semitopological semigroup for any non-negative integers $m$ and $n$.
\end{remark}

The following theorem generalized Theorem~1 from \cite{Gutik-2015}.

\begin{theorem}\label{theorem-8}
Let $m$ and $n$ be arbitrary non-negative integers. If $\left(\mathscr{C}_{m,n}^0,\tau\right)$ is a Hausdorff locally compact semitopological semigroup, then $\tau$ is either discrete or $\tau=\tau_{\operatorname{\textsf{Ac}}}$.
\end{theorem}

\begin{proof}
Let $\tau$ be a Hausdorff locally compact topology on $\mathscr{C}_{m,n}^0$ such that $\left(\mathscr{C}_{m,n}^0,\tau\right)$ is a semitopological semigroup and the zero $0$ of $\mathscr{C}_{m,n}^0$ is not an isolated point of the space $\left(\mathscr{C}_{m,n}^0,\tau\right)$. By Lemma~\ref{lemma-1} the subsemigroup $\mathscr{C}_{m,n}^*$ of $\mathscr{C}_{m,n}$ is isomorphic to the bicyclic semigroup $\mathscr{C}(p,q)$ and hence the subsemigroup $\left(\mathscr{C}_{m,n}^*\right)^0=\mathscr{C}_{m,n}^*\sqcup\{0\}$ of $\mathscr{C}_{m,n}^0$ is isomorphic to the bicyclic semigroup with adjoined zero $\mathscr{C}^0=\mathscr{C}(p,q)\sqcup\{0\}$. Theorem~\ref{theorem-3} implies that $\mathscr{C}_{m,n}$ is a dense discrete subspace of $\left(\mathscr{C}_{m,n}^0,\tau\right)$, so it is open
by Corollary~3.3.10 of \cite{Engelking-1989}. This Corollary also implies that  the subspace $\left(\mathscr{C}_{m,n}^*\right)^0$ of $\left(\mathscr{C}_{m,n}^0,\tau\right)$ is locally compact.

We claim that for every open neighbourhood $V(0)$ of zero $0$ in $\left(\mathscr{C}_{m,n}^0,\tau\right)$ the set $V(0)\cap\left(\mathscr{C}_{m,n}^*\right)^0$ is infinite. Suppose to the contrary that there exists an open neighbourhood $V(0)$ of zero $0$ in $\left(\mathscr{C}_{m,n}^0,\tau\right)$ such that the set $V(0)\cap\left(\mathscr{C}_{m,n}^*\right)^0$ is finite. Since the space $\left(\mathscr{C}_{m,n}^0,\tau\right)$ is Hausdorff, without loss of generality we may assume that $V(0)\cap\left(\mathscr{C}_{m,n}^*\right)^0=\{0\}$. Then by the separate continuity of the semigroup operation of $\left(\mathscr{C}_{m,n}^0,\tau\right)$ there exists an open neighbourhood $W(0)$ of zero in $\left(\mathscr{C}_{m,n}^0,\tau\right)$ such that $W(0)\subseteq V(0)$ and
\begin{equation*}
  \left(q^np^m\ast_{m,n}W(0)\right)\cup \left(W(0)\ast_{m,n}q^np^m\right)\subseteq V(0).
\end{equation*}
Since $0$ is a non-isolated point of $\left(\mathscr{C}_{m,n}^0,\tau\right)$, at least one of the following conditions holds:
\begin{itemize}
  \item[(a)] the set $W(0)\cap\left\{q^ip^j\colon i\in\mathbb{N}_0, j=0,1,\ldots,m-1\right\}$ is infinite;
  \item[(b)] the set $W(0)\cap\left\{q^ip^j\colon i=0,1,\ldots,n-1, j\in\mathbb{N}_0\right\}$ is infinite.
\end{itemize}

If (a) holds then the neighbourhood $W(0)$ contains infinitely many elements of the form $q^ip^j$, where $j<m$, for which we have that
\begin{equation*}
  q^ip^j\ast_{m,n}q^np^m=q^ip^jq^mp^nq^np^m=q^ip^jq^mp^m=q^{i-j+m}p^m\in\mathscr{C}_{m,n}^*.
\end{equation*}
Similarly, if (b) holds then the neighbourhood $W(0)$ contains infinitely many elements of the form $q^ip^j$, where $i<n$, for which we have that
\begin{equation*}
  q^np^m\ast_{m,n}q^ip^j=q^np^mq^mp^nq^ip^j=q^np^nq^ip^j=q^np^{n-i+j}\in\mathscr{C}_{m,n}^*.
\end{equation*}
The above arguments imply that the set $V(0)\cap\left(\mathscr{C}_{m,n}^*\right)^0$ is infinite. Hence we have that the zero $0$ is a non-isolated point in the subspace $\left(\mathscr{C}_{m,n}^*\right)^0$ of $\left(\mathscr{C}_{m,n}^0,\tau\right)$.

By Lemma~\ref{lemma-1} the subsemigroup $\mathscr{C}_{m,n}^*$ of $\mathscr{C}_{m,n}$ is isomorphic to the bicyclic semigroup and hence
by Theorem~1 from \cite{Gutik-2015} we obtain that the space $\left(\mathscr{C}_{m,n}^*\right)^0$ is compact. Then for every open neighbourhood $U(0)$ of the zero $0$ in $\left(\mathscr{C}_{m,n}^0,\tau\right)$ we have that the set $\left(\mathscr{C}_{m,n}^*\right)^0\setminus U(0)$ is finite.

Now, the semigroup operation of $\mathscr{C}_{m,n}^0$ implies that
\begin{equation*}
  p^m\ast_{m,n}q^ip^j=p^mq^mp^nq^ip^j=p^nq^ip^j=q^{i-n}p^j
\end{equation*}
and
\begin{equation*}
  q^ip^j\ast_{m,n}q^n=q^ip^jq^mp^nq^n=q^ip^jq^m=q^ip^{j-m},
\end{equation*}
for arbitrary element $q^ip^j\in\mathscr{C}_{m,n}^*$.
This and the definition of $\mathscr{C}_{m,n}^*$ imply that
\begin{equation*}
  p^m\ast_{m,n}\mathscr{C}_{m,n}^*=\left\{q^{i-n}p^j\colon i\geqslant n, j\geqslant m\right\}
\end{equation*}
and
\begin{equation*}
  \mathscr{C}_{m,n}^*\ast_{m,n}q^n=\left\{q^ip^{j-m}\colon i\geqslant n, j\geqslant m\right\}.
\end{equation*}
Thus the set  $\mathscr{C}_{m,n}^0\setminus\left(p^m\ast_{m,n}\left(\mathscr{C}_{m,n}^*\right)^0\cup\left(\mathscr{C}_{m,n}^*\right)^0\ast_{m,n}q^n\right)$ is finite, and hence the above arguments imply that every open neighbourhood $U(0)$ of the zero $0$ in $\left(\mathscr{C}_{m,n}^0,\tau\right)$ has a finite complement in the space $\left(\mathscr{C}_{m,n}^0,\tau\right)$. Thus the space $\left(\mathscr{C}_{m,n}^0,\tau\right)$ is compact and by Remark~\ref{remark-8} the semitopological semigroup $\mathscr{C}_{m,n}^0$ is topologically isomorphic to the semitopological semigroup $\left(\mathscr{C}_{m,n}^0,\tau_{\operatorname{\textsf{Ac}}}\right)$.
\end{proof}

Since by Corollary~\ref{corollary-5} the interassociate $\mathscr{C}_{m,n}$ of the bicyclic monoid ${\mathscr{C}}(p,q)$ does not embed into any Hausdorff compact topological semigroup, Theorem~\ref{theorem-8} implies the following corollary.

\begin{corollary}\label{corollary-9}
If $m$ and $n$ are arbitrary non-negative integers and $\mathscr{C}_{m,n}^0$ is a Hausdorff locally compact topological semigroup, then $\mathscr{C}_{m,n}^0$ is discrete.
\end{corollary}

The following example shows that a counterpart of the statement of Corollary~\ref{corollary-9} does not hold when $\mathscr{C}_{m,n}^0$ is a \v{C}ech-complete metrizable topological semigroup for any non-negative integers $m$ and $n$.

\begin{example}\label{example-10}
Fix arbitrary non-negative integers $m$ and $n$. On the semigroup $\mathscr{C}_{m,n}^0$ we define a topology $\tau_{\operatorname{\textsf{1}}}$ in the following way:
\begin{itemize}
  \item[$(i)$] every element of the interassociate $\mathscr{C}_{m,n}$ of the bicyclic monoid is an isolated point in the space $(\mathscr{C}_{m,n}^0,\tau_{\operatorname{\textsf{1}}})$;
  \item[$(ii)$] the family $\mathscr{B}_1(0)=\left\{U_s \colon s\in \mathbb{N}_0\right\}$, where
   \begin{equation*}
   U_s=\{0\}\cup\left\{q^{n+i}p^{m+j}\in\mathscr{C}_{m,n}^0\colon i,j>s\right\},
\end{equation*}
is a base of the topology $\tau_{\operatorname{\textsf{1}}}$ at the zero.
\end{itemize}
It is obvious that $(\mathscr{C}_{m,n}^0,\tau_{\operatorname{\textsf{1}}})$ is first countable. Then the definition of the semigroup operation of $\mathscr{C}_{m,n}^0$ and the arguments presented in \cite[p.~68]{Gutik-1996} show that $(\mathscr{C}_{m,n}^0,\tau_{\operatorname{\textsf{1}}})$ is a Hausdorff topological semigroup.

First we observe that each element of the family $\mathscr{B}_1(0)$ is an open-and-closed subset of $(\mathscr{C}_{m,n}^0,\tau_{\operatorname{\textsf{1}}})$, and hence the space $(\mathscr{C}_{m,n}^0,\tau_{\operatorname{\textsf{1}}})$ is regular. Since the space $\mathscr{C}_{m,n}^0$ is countable and first countable, it is second countable and hence by Theorem~4.2.9 from \cite{Engelking-1989} it is metrizable. Also, by Theorem~4.3.26 from \cite{Engelking-1989} the space $(\mathscr{C}_{m,n}^0,\tau_{\operatorname{\textsf{1}}})$ is \v{C}ech-complete, as a completely metrizable space.
\end{example}

Also the following example presents an interassociate of the bicyclic semigroup with adjoined zero $\mathscr{C}^0=\mathscr{C}(p,q)\sqcup\{0\}$ for which a counterpart of the statements of Theorem~\ref{theorem-8} and of Corollary~\ref{corollary-9} do not hold.

\begin{example}\label{example-11}
The interassociate of the bicyclic semigroup with adjoined zero $\mathscr{C}^0$ with the operation $a\ast b=a\cdot 0\cdot b$ is a countable semigroup with zero-multiplication. It is well known that this semigroup endowed with any topology is a topological semigroup (see \cite[Vol.~1, Chapter~1]{Carruth-Hildebrant-Koch-1983-1986}).
\end{example}

Later we shall need the following notions. A continuous map $f\colon X\to Y$ from a topological space $X$ into a topological space $Y$ is called:
\begin{itemize}
  \item[$\bullet$] \emph{quotient} if the set $f^{-1}(U)$ is open in $X$ if and only if $U$ is open in $Y$ (see \cite{Moore-1925} and \cite[Section~2.4]{Engelking-1989});
  \item[$\bullet$] \emph{hereditarily quotient} or \emph{pseudoopen} if for every $B\subset Y$ the restriction $f|_{B}\colon f^{-1}(B)$ $\rightarrow B$ of $f$ is a quotient map (see \cite{McDougle-1958, McDougle-1959, Arkhangelskii-1963} and \cite[Section~2.4]{Engelking-1989});
  \item[$\bullet$] \emph{closed} if $f(F)$ is closed in $Y$ for every closed subset $F$ in $X$;
  \item[$\bullet$] \emph{perfect} if $X$ is Hausdorff, $f$ is a closed map and all fibers $f^{-1}(y)$ are compact subsets of $X$ (see \cite{Vainstein-1947} and \cite[Section~3.7]{Engelking-1989}).
\end{itemize}
Every closed map and every hereditarily quotient map are quotient \cite{Engelking-1989}. Moreover, a continuous map $f\colon X\to Y$ from a topological space $X$ onto a topological space $Y$ is hereditarily quotient if and only if for every $y\in Y$ and every open subset $U$ in $X$ which contains $f^{-1}(y)$ we have that $y\in\operatorname{int}_Y(f(U))$ (see \cite[2.4.F]{Engelking-1989}).

We need the following trivial lemma, which follows from separate continuity of the semigroup operation in semitopological semigroups.

\begin{lemma}\label{lemma-11}
Let $S$ be a Hausdorff semitopological semigroup and $I$ be a compact ideal in $S$. Then the Rees-quotient semigroup $S/I$ with the quotient topology is a Hausdorff semitopological semigroup.
\end{lemma}

The following theorem generalizes Theorem~2 from \cite{Gutik-2015}.

\begin{theorem}\label{theorem-12}
Let $(\mathscr{C}_{m,n}^I,\tau)$ be a Hausdorff locally compact semitopological semigroup, $\mathscr{C}_{m,n}^I=\mathscr{C}_{m,n}\sqcup I$ and $I$ is a compact ideal of $\mathscr{C}_{m,n}^I$. Then either $(\mathscr{C}_{m,n}^I,\tau)$ is a compact semitopological semigroup or the ideal $I$ is open.
\end{theorem}

\begin{proof}
Suppose that $I$ is not open. By Lemma~\ref{lemma-11} the Rees-quotient semigroup $\mathscr{C}_{m,n}^I/I$ with the quotient topology $\tau_{\operatorname{\textsf{q}}}$ is a semitopological semigroup. Let $\pi\colon \mathscr{C}_{m,n}^I\to \mathscr{C}_{m,n}^I/I$ be the natural homomorphism, which is a quotient map. It is obvious that the Rees-quotient semigroup $\mathscr{C}_{m,n}^I/I$ is isomorphic to the semigroup $\mathscr{C}_{m,n}^0$ and the image $\pi(I)$ is zero of $\mathscr{C}_{m,n}^I/I$. Now we shall show that the natural homomorphism $\pi\colon \mathscr{C}_{m,n}^I\to \mathscr{C}_{m,n}^I/I$ is a hereditarily quotient map. Since $\pi(\mathscr{C}_{m,n})$ is a discrete subspace of $(\mathscr{C}_{m,n}^I/I,\tau_{\operatorname{\textsf{q}}})$, it is sufficient to show that for every open neighbourhood $U(I)$ of the ideal $I$ in the space $(\mathscr{C}_{m,n}^I,\tau)$ the image  $\pi(U(I))$ is an open neighbourhood of the zero $0$ in the space $(\mathscr{C}_{m,n}^I/I,\tau_{\operatorname{\textsf{q}}})$. Indeed, $\mathscr{C}_{m,n}^I\setminus U(I)$ is an open subset of $(\mathscr{C}_{m,n}^I,\tau)$, because the elements of the semigroup $\mathscr{C}_{m,n}$ are isolated points of the space $(\mathscr{C}_{m,n}^I,\tau)$. Also, since the restriction $\pi|_{\mathscr{C}_{m,n}}\colon \mathscr{C}_{m,n}\to \pi(\mathscr{C}_{m,n})$ of the natural homomorphism $\pi\colon \mathscr{C}_{m,n}^I\to \mathscr{C}_{m,n}^I/I$ is one-to-one,  $\pi(\mathscr{C}_{m,n}^I\setminus U(I))$ is a closed subset of $(\mathscr{C}_{m,n}^I/I,\tau_{\operatorname{\textsf{q}}})$. So $\pi(U(I))$ is an open neighbourhood of the zero $0$ of the semigroup $(\mathscr{C}_{m,n}^I/I,\tau_{\operatorname{\textsf{q}}})$, and hence the natural homomorphism $\pi\colon \mathscr{C}_{m,n}^I\to \mathscr{C}_{m,n}^I/I$ is a hereditarily quotient map. Since $I$ is a compact ideal of the semitopological semigroup $(\mathscr{C}_{m,n}^I,\tau)$, $\pi^{-1}(y)$ is a compact subset of $(\mathscr{C}_{m,n}^I,\tau)$ for every $y\in \mathscr{C}_{m,n}^I/I$. By Din' N'e T'ong's Theorem (see \cite{Din'-N'e-T'ong-1963} or \cite[3.7.E]{Engelking-1989}), $(\mathscr{C}_{m,n}^I/I,\tau_{\operatorname{\textsf{q}}})$ is a Hausdorff locally compact space. Since $I$ is not open, by Theorem~\ref{theorem-8} the semitopological semigroup $(\mathscr{C}_{m,n}^I/I,\tau_{\operatorname{\textsf{q}}})$ is topologically isomorphic to $(\mathscr{C}_{m,n}^0,\tau_{\operatorname{\textsf{Ac}}})$ and hence it is compact.
We claim that the space $(\mathscr{C}_{m,n}^I,\tau)$ is compact. Indeed, let $\mathscr{U}=\left\{U_\alpha\colon\alpha\in\mathscr{I}\right\}$ be an arbitrary open cover of the topological space $(\mathscr{C}_{m,n}^I,\tau)$. Since $I$ is compact, there exists a finite family $\{U_{\alpha_1},\ldots,U_{\alpha_n}\}\subset\mathscr{U}$ such that $I\subseteq U_{\alpha_1}\cup\cdots\cup U_{\alpha_n}$. Put $U=U_{\alpha_1}\cup\cdots\cup U_{\alpha_n}$. Then $\mathscr{C}_{m,n}^I\setminus U$ is a closed-and-open subset of $(\mathscr{C}_{m,n}^I,\tau)$. Also, since the restriction $\pi|_{\mathscr{C}_{m,n}}\colon \mathscr{C}_{m,n}\to \pi(\mathscr{C}_{m,n})$ of the natural homomorphism $\pi$ is one-to-one, $\pi(\mathscr{C}_{m,n}^I\setminus U(I))$ is an open-and-closed subset of $(\mathscr{C}_{m,n}^I/I,\tau_{\operatorname{\textsf{q}}})$, and hence the image $\pi(\mathscr{C}_{m,n}^I\setminus U(I))$ is finite, because the semigroup $(\mathscr{C}_{m,n}^I/I,\tau_{\operatorname{\textsf{q}}})$ is compact. Thus, the set $\mathscr{C}_{m,n}^I\setminus U$ is finite as well and hence the space $(\mathscr{C}_{m,n}^I,\tau)$ is also compact.
\end{proof}

\begin{corollary}\label{corollary-13}
If $(\mathscr{C}_{m,n}^I,\tau)$ is a Hausdorff locally compact topological semigroup, $\mathscr{C}_{m,n}^I=\mathscr{C}_{m,n}\sqcup I$ and $I$ is a compact ideal of $\mathscr{C}_{m,n}^I$, then the ideal $I$ is open.
\end{corollary}

\section*{Acknowledgements}

We acknowledge the referee for useful important comments and suggestions.

\end{document}